\documentclass[12pt]{amsart}
\usepackage{amsmath,amsfonts,amssymb,amsthm, a4wide, url, mathscinet,mathdots, bbm}
\usepackage
[colorlinks = true,
            linkcolor = blue,
            urlcolor  = blue,
            citecolor = green,
            anchorcolor = blue]{hyperref}
\usepackage[usenames]{color}
\usepackage[nobysame,initials]{amsrefs}
\usepackage{xcolor}
\newtheorem{theorem}{Theorem}[section]
\newtheorem{lemma}[theorem]{Lemma}
\newtheorem{claim}[theorem]{Claim}

\newtheorem{proposition}[theorem]{Proposition}
\newtheorem{corollary}[theorem]{Corollary}
\newtheorem{conjecture}[theorem]{Conjecture}
\newtheorem{remark}[theorem]{Remark}
\newtheorem{obs}[theorem]{Observation}
\newtheorem{problema}[theorem]{Problem}
\newtheorem*{notat*}{Notation}

\let\originalleft\left
\let\originalright\right
\renewcommand{\left}{\mathopen{}\mathclose\bgroup\originalleft}
\renewcommand{\right}{\aftergroup\egroup\originalright}

\numberwithin{equation}{section}

\global\long\def\conv#1{{\rm conv}\left(#1\right)}%

\global\long\def\PP{{\rm Pr}}%

\global\long\def\iprod#1#2{\langle#1,\,#2\rangle}%

\global\long\def\d{{\,\rm d}}%

\global\long\def\sub{\subset}%

\global\long\def\vol#1{{\rm vol}\left(#1\right)}

\global\long\def\supp#1{{\rm supp}\left(#1\right)}%

\global\long\def\RR{\mathbb{R}}%

\global\long\def\EE{\mathbb{E}}%

\global\long\def\S{\mathbb{S}}%

\global\long\def\volco#1{{\rm vol_{n-1}}\left(#1\right)}%

\def\ec#1#2{{\EE{[#1\,|\,#2]}}}%

\title{Isobarycentric Inequalities}
\author{Shoni Gilboa}
\address{Department of Mathematics, The Open University of Israel, Raanana 4353701, Israel}

\author{Pazit Haim-Kislev}
\address{Department of Mathematics, Tel Aviv University, Tel Aviv 6997801, Israel}

\author{Boaz Slomka}
\address{Department of Mathematics, The Open University of Israel, Raanana 4353701, Israel}

\begin{document}
\begin{abstract}
    We prove the following isoperimetric type inequality: Given a finite absolutely continuous Borel measure on $\RR^n$, halfspaces have maximal measure among all subsets with prescribed barycenter. As a consequence, we make progress towards a solution to a problem of Henk and Pollehn, which is equivalent to  a Log-Minkowski inequality for a parallelotope and a centered convex body. Our probabilistic approach to the problem also gives rise to several inequalities and conjectures concerning the truncated mean of certain log-concave random variables. 
\end{abstract}

\maketitle
 
\section{Introduction}

\subsection{An isobarycentric problem}

Which subset of the unit cube $[-1,1]^n$ in $\RR^n$ has maximal volume among all subsets with prescribed barycenter (center of mass)? 
This question is closely related to a problem posed by Martin Henk, which we discuss in the next section.

More generally, given a finite Borel measure $\mu$ on $\RR^n$, we consider the following problem, which we refer to as the {\em isobarycentric problem for $\mu$}: which subsets of $\RR^n$ maximize $\mu$ among all subsets whose barycenter with respect to $\mu$ is some fixed point $x$ in its support?

When $\mu$ is the uniform measure on the unit Euclidean ball in $\RR^n$ or, say, the standard Gaussian measure on $\RR^n$, it is perhaps expected that the answer should be a halfspace whose barycenter is $x$. Our first result shows that this holds in greater generality, under mild assumptions on the underlying measure. 

Denote the barycenter of a subset $A\sub\RR^n$ with respect to $\mu$  by $b_\mu(A)=\frac{1}{\mu(A)}\int x\d\mu(x)$ and the support of $\mu$ by $\supp{\mu}$. For a set $A\sub \RR^n$ denote by ${\rm int}(A)$  and $\conv{A}$ its interior and  convex hull.
\begin{theorem}\label{thm:main1}
	Let $\mu$ be a finite absolutely continuous Borel measure on $\RR^n$ with bounded first moment. Then for every $b_\mu(\RR^n)\neq x\in{\rm int}(\conv{\supp{\mu})}$ there exists a halfspace $\mathcal{H}\sub\RR^n$ which maximizes $\mu$ among all subsets whose barycenter is $x$. That is,  $b_\mu(\mathcal{H})=x$ and $ \mu(\mathcal{H})=\sup\{\mu(A)\,:\, b_\mu(A)=x\}$. Moreover, if $\partial \mathcal{H} \cap {\rm int}(\supp{\mu}) \neq \emptyset$ then $\mathcal{H}$ is the unique halfspace for which $b_\mu(\mathcal{H}) = x$.
\end{theorem}

\subsection{Relation to the Log-Minkowski inequality}

One  application of Theorem \ref{thm:main1} is verifying the Log-Minkowski inequality in certain cases (see Corollary \ref{cor:log-mink} below).
The Log-Minkowski inequality was conjectured by B\"{o}r\"{o}czky,  Lutwak,  Yang,  and Zhang in \cite{BLYZ} and states that  for any two centrally-symmetric convex bodies $K,L\sub\RR^n$ one has 
\begin{align}
		\label{eq:Log-BM}
		\int_{S^{n-1}}\log\frac{h_K(u)}{h_L(u)}\d V_L(u)\geq\frac{{\rm vol}(L)}{n}\log\frac{{\rm vol}(K)}{{\rm vol}(L)},
	\end{align}
where $V_L$ is the cone-volume measure of $L$ and $h_B(u)$ denotes the support function of a convex body $B\sub\RR^n$ in direction $u\in\S^{n-1}$.
	
In the same paper, the Log-Brunn-Minkowski inequality was also conjectured and shown to be equivalent to the Log-Minkowski inequality. If true, these inequalities would make a far-reaching extension to the classical Brunn-Minkowski inequality, see e.g., \cite{AGM}.

So far, the Log-Minkowski inequality has been verified  for all planar convex bodies in \cite{BLYZ}, for the class of unconditional bodies in \cite{Saroglou} and for several other families of convex bodies - some of which are not centrally-symmetric - in \cite{Stancu} and \cite{HP}. 
We remark that  the Log-Minkowski inequality does not hold for arbitrary non-symmetric convex bodies, as pointed out in \cite{BLYZ}.

The following quantitative isobarycentric problem was posed by Martin Henk, using a different normalization, and appears implicitly in \cite{HP}. It is this problem that sparked the research presented here.

Denote the Lebesgue volume of a measurable set $A\sub\RR^{n}$
by $\vol A$, and its barycenter  by $b\left(A\right)=\int_A x \d x/\vol A$.
If $b(A)$ is at the origin, we say that $A$ is \emph{centred}. 

\begin{problema}
	\label{prob:main}
	Is it true that for every non-zero $x=(x_1,\dots,x_n)\in(-1,1)^n$ and each $K\sub[-1,1]^n$ such that $b(K)=x$, one has 
	\begin{equation}
		\vol{K}< 2^n\prod_{i=1}^n\sqrt{1-x_i^2}\,\,?\label{eq:HP}
	\end{equation}
\end{problema}
\noindent Note that for $x=0$, \eqref{eq:HP} becomes equality if $K=[-1,1]^n$. In its original formulation, the set $K$  in Problem \ref{prob:main} was assumed to be convex. However, in view of Theorem \ref{thm:main1}, this additional assumption is redundant. 

As observed in \cite{HP}, an affirmative answer to Problem \ref{prob:main} would imply the Log-Minkowski inequality for a centered convex body and a centered parallelotope. An affirmative answer to this problem was confirmed in \cite{HP}*{Proposition 1.5} for $n\le4$. 

We remark that an affirmative answer to Problem \ref{prob:main} in a given dimension $n_0$ would imply an affirmative answer for all $n\le n_0$ via a simple tensoring argument, mapping $K\sub\RR^n$ to $K\times [-1,1]^{n_0-n}\sub\RR^{n_0}$.

As an application of Theorem \ref{thm:main1}, we resolve Problem \ref{prob:main} in the case where $b(K)$ is on the main diagonal of the cube. Let $\mathbf{1}_n$ denote the vector $(1,1,\ldots,1)$ in $\RR^n$.

\begin{theorem}
	\label{thm:cube_main}
	For every non-zero $-1<m<1$ and each $K \subset [-1,1]^n$ such that $b(K) = m\mathbf{1}_n$, 
	one has $$ \sqrt[n]{\vol{K}}< 2 \sqrt{1-m^2}.$$ 
\end{theorem}

As one might expect, Theorem \ref{thm:cube_main} can be applied in order to obtain the Log-Minkowski inequality for a specific family of centered convex bodies. Indeed, adapting the proof of \cite{HP}*{Proposition 1.5} and invoking Theorem \ref{thm:cube_main} yields:
\begin{corollary} \label{cor:log-mink}
	Suppose $Q=-Q \subset \RR^n$ is a parallelotope whose outer normals are $\{\pm u_1,\ldots,\pm u_n\}$ and $K \subset \RR^n$ is a centered convex body such that $h_K(u_i)/h_K(-u_i)$ 
	is independent of $i$. Then \eqref{eq:Log-BM} holds. 
\end{corollary}

\subsection{A probabilistic approach}
Let $U_1,\dots,U_n$ be independent uniform random variables on $[-1,1]$ and denote by $U$ the uniformly distributed random vector $(U_1,\dots,U_n)$ in $[-1,1]^n$.
In probabilistic terms, Problem \ref{prob:main} reads as follows: is it true that for every measurable set $K \subset [-1,1]^n$ such that $0 < \vol{K} < 2^n$, it holds that
$$ 
\PP(U\in K) < \prod_{i=1}^n \sqrt{1-\ec{U_i}{{U \in K}}^2}\,\,?
$$
In view of Theorem \ref{thm:main1}, the condition for $U$ to be in $K$ can be simplified. 
In particular, Theorem \ref{thm:cube_main} takes an even simpler form, as every point $m\mathbf{1}_n$ on the main diagonal of the cube $[-1,1]^n$ is the barycenter of a (unique) intersection of the cube with a halfspace orthogonal to $\mathbf{1}_n$. Namely, one can verify that Theorem \ref{thm:cube_main} is equivalent to showing that for $X=\frac{1}{n}\sum_{i=1}^nU_i$ and all $-1<t<1$, one has

\begin{equation} \label{eq:diag} 
\sqrt[n]{\PP(X>t)}<\sqrt{1-m_X(t)^2}, 
\end{equation} 
where 
$$m_X(t):=\ec{X}{X>t}$$ 
is the {\em truncated mean} of $X$.

In addition to providing a convenient framework for the proof of Theorem \ref{thm:cube_main}, the probabilistic formulation of Problem \ref{prob:main} also gives rise to straightforward extensions of it. 
For example, differentiating \eqref{eq:diag} yields the following (stronger) curious inequality (see Lemma \ref{lem:deriv} for the details).
\begin{equation}\label{eq:deriv_enough}
\left(m_X(t)-t\right)\frac{m_X(t)}{1-m_X(t)^2}<\frac{1}{n}.
\end{equation}
In fact, due to numerical experiments, we believe that a stronger inequality holds:

\begin{conjecture}
\label{conj:derivative}
Suppose $X = \frac{1}{n} \sum_{i=1}^n U_i$, where $U_1,\ldots,U_n$ are uniform random variables on $[-1,1]$. Then for every $-1<t<1$,
\begin{equation*}
\left(m_X(t)-t\right)\frac{m_X(t)}{1-m_X(t)^2}<\frac{1}{2n}.
\end{equation*}
\end{conjecture}

We verify the conjecture for $n=1$ and a wider class of random variables. We also obtain  partial results supporting its validity for higher values of $n$ (see Section \ref{sec:derived}). Some of our partial results concerning Conjecture \ref{conj:derivative} are also essential for the proof of Theorem \ref{thm:cube_main}. 

It is  worth mentioning that the properties of truncated means of log-concave random variables, such as in \eqref{eq:diag} and Conjecture  \ref{conj:derivative}, occasionally play an important role in Econometrics and Economics of uncertainity  theory, see e.g., \cite{B96}, \cite{AAV} and references therein. 

\subsection*{Structure of the paper} 
The paper is organized as follows. In Section \ref{sec:extreme} we prove Theorem \ref{thm:main1}. In Section \ref{sec:derived} we discuss the inequality \eqref{eq:deriv_enough} and lay the groundwork for the proof of Theorem  \ref{thm:cube_main}. The proof is completed in Section \ref{sec:iso_cube}, which also contains another general isobarycentric inequality for the cube. In Section \ref{sec:discussion} we conclude the paper with suggestions for some directions for future study.

\subsection*{Acknowledgement}
We are grateful to  Martin Henk for sharing Problem \ref{prob:main} with us and explaining its context,  as well as for his useful comments and suggestions. We thank Alon Nishry for his help with numerical experiments which resulted in the formulation of Conjecture \ref{conj:derivative}. We also thank Shiri Artstein-Avidan, Bo'az Klartag and Yaron Ostrover for fruitful discussions. 
The second named author is partially supported by the European Research Council grant No. 637386.
The third named author is supported by ISF grant 784/20.

 \section{Characterization of extremal centered sets}\label{sec:extreme}

Our goal in this section is to prove Theorem \ref{thm:main1}. 
We start with some assisting lemmas in the case where $\supp{\mu}$ is bounded and connected.
The first lemma proves that if there exists a half-space with the required barycenter then it maximizes the volume out of all the domains with the same barycenter.
We continue with showing that no two distinct half-spaces share the same barycenter, and then prove the existence of the required half-space.
We conclude by approximating all measures by measures whose support is bounded and connected.

Let $\iprod{\cdot}{\cdot}$ denote the standard Euclidean scalar product on $\RR^{n}$. Given a finite absolutely continuous Borel measure $\mu$ on $\RR^n$, a measurable subset $A\sub\RR^n$ and $\theta\in\S^{n-1}$, the coordinate of $b_\mu\left(A\right)$
in direction $\theta$ is given by $b_{\mu,{\theta}}\left(A\right)=\int_{A}\iprod x{\theta}\d \mu/\mu(A)$.

Given a direction $\theta$ and $c\in\RR$, denote the hyperplane ${H}_{\theta,c}=\left\{ x\,:\,\iprod x{\theta}=c\right\}$ and the halfspaces $H^-_{\theta,c}=\left\{ x\,:\,\iprod x{\theta}\le c\right\}$ and $H^+_{\theta,c}=\left\{ x\,:\,\iprod x{\theta}\ge c\right\}$.  For $A,B\sub\RR^n$, denote by $A^c$ and $A\triangle B$ the complement and the symmetric difference operations.

Throughout this section, we shall make an extensive use of the following simple property of the barycenter $b_\mu$: For every two Borel sets $A,B\sub\RR^n$ such that $\mu(A\cap B)=0$ one has
\begin{equation}\label{eq:LTE}
    \mu(A\cup B)b_\mu (A\cup B)=\mu(A)b_\mu (A)+\mu(B)b_\mu (B),
\end{equation}
where our convention is that $\mu(X) b_{\mu}(X) = 0$ whenever $\mu(X) = 0$.
\begin{lemma}
\noindent \label{lem::maximizer}
Let $H_{\theta,c}$ be a hyperplane.
Then for every $A\subset \RR^n$
with $b_{\mu,\theta}\left(A\right) = b_{\mu,\theta}\left(H_{\theta,c}^{-}\right)$, we have $\mu(A) \le \mu(H_{\theta,c}^{-})$, with equality only when $\mu(H_{\theta,c}^- \triangle A)=0$. In particular, if $b_\mu\left(H_{\theta,c}^{-}\right) = b_\mu(A)$ then $\mu(A) \leq \mu\left(H_{\theta,c}^{-}\right)$, with equality only when $\mu(H_{\theta,c}^- \triangle A)=0$.
\end{lemma}

\begin{proof}
For brevity, in this proof  we set $b_{\theta}(\cdot)=b_{\mu,\theta}(\cdot)$, $H=H_{\theta,c}$ and $H^{\pm}=H^\pm_{\theta,c}$. 
Without loss of generality, we assume that $b_{\theta}(H^-) = b_{\theta}(A)=0$.
From the definition of $H^-$ and the fact that $b_{\theta}(H^-)=0$ we deduce that $c>0$.
On the one hand, we have 
\begin{equation*}
0=b_{\theta}\left(H^-\right)\mu(H^-)=\mu(H^-\cap A)b_{\theta}\left(H^-\cap A\right)+\mu(H^-\cap A^{c})b_{\theta}\left(H^-\cap A^{c}\right).
\end{equation*}
On the other hand, we have
\begin{equation*}
0=b_{\theta}\left(A\right)\mu(A)=\mu(A\cap H^-)b_{\theta}\left(A\cap H^-\right)+\mu(A\cap H^+)b_{\theta}\left(A\cap H^+\right).
\end{equation*}
In particular, it follows that 
\begin{align}
\label{eq::two_sides_eq}
\mu(H^-\cap A^{c})b_{\theta}\left(H^-\cap A^{c}\right)=\mu(A\cap H^+)b_{\theta}\left(A\cap H^+\right).
\end{align}
Note that if $\mu(A\cap H^+) = 0$ then indeed $\mu(A) \leq \mu(H^-)$ with equality if and only if $\mu(A \triangle H^-) = 0$. If $\mu(H^- \cap A^c) = 0$ then $b_{\theta}(H^- \cap A) = b_{\theta}(H^-) = 0$ which gives $\mu(A \cap H^+) = 0$, since otherwise $b_{\theta}(A \cap H^+) = 0$, which contradicts $c>0$.
Hence we may assume that $\mu(A\cap H^+)$, $\mu(H^- \cap A^c) \neq 0$ and $b_{\theta}(A \cap H^+), b_{\theta}(H^- \cap A^{c})$ are well defined.
The inequalities $b_{\theta}\left(H^-\cap A^{c}\right) < c < b_{\theta}\left(H^+\cap A\right)$
imply that
\[
\mu(H^+\cap A)=\frac{b_{\theta}\left(H^-\cap A^{c}\right)}{b_{\theta}\left(H^+\cap A\right)}\mu(H^-\cap A^{c}) < \mu(H^-\cap A^{c})
\]
which gives $\mu(A) < \mu(H^-)$.
\end{proof}
\begin{lemma}
\label{lem::uniqeness}
Suppose $H_{1}=H_{\theta_1,c_1} ,H_{2}=H_{\theta_2,c_2}$
are distinct hyperplanes, both containing interior points of $\supp{\mu}$. Then $b_\mu(H_{1}^{-})\neq b_\mu(H_{2}^{-})$.
\end{lemma}

\begin{proof}
Assume by contradiction that $u=b_\mu (H_{1}^{-})=b_\mu(H_{2}^{-})$.
By translating everything by $-u$, we may assume that $u=0$. 
Note that our condition on the hyperplanes guarantees that the measure of each half-space is positive and $c_1,c_2>0$.
We have
\[
0=b(H_1^- \cap H_2^- ) \mu(H_1^-\cap H_2^-)+b(H_1^-\cap H_2^+)\mu(H_1^-\cap H_2^+),
\]
\[
0=b(H_1^- \cap H_2^- ) \mu(H_1^-\cap H_2^-)+b(H_1^+\cap H_2^-)\mu(H_1^+\cap H_2^-).
\]
Recall that our convention is that $\mu(X)b(X)=0$ if $\mu(X)=0$.
If $\mu(H_1^- \cap H_2^+) = 0 = \mu(H_1^+ \cap H_2^-)$, we must have $\mu(H_1^+ \cap H_2^+), \mu(H_1^- \cap H_2^-) > 0$, which contradicts the connectedness of $\supp{\mu}$.
If $\mu(H_1^- \cap H_2^+) = 0$, since $b(H_1^+ \cap H_2^-) \geq c_1 > 0$, we get $\mu(H_1^+ \cap H_2^-) = 0$, which is a contradiction.
Same holds for the case $\mu(H_1^+ \cap H_2^-) = 0$.
Therefore, $b(H_1^- \cap H_2^+) =\alpha b(H_1^+ \cap H_2^-)$ for some $\alpha>0$. 
However it is a contradiction since 
$$
b_{\theta_1}(H_1^- \cap H_2^+) < c_1 < b_{\theta_1}(H_1^+ \cap H_2^-) = \frac{1}{\alpha} b_{\theta_1}(H_1^- \cap H_2^+) $$
\begin{equation*}
b_{\theta_2}(H_1^+ \cap H_2^-) < c_2 < b_{\theta_2}(H_1^- \cap H_2^+) = \alpha b_{\theta_2}(H_1^+ \cap H_2^-).  
\qedhere\end{equation*}
\end{proof}

\begin{lemma}
\label{lem::exsitence}
Let $\mu$ be a finite absolutely continuous Borel measure  on $\RR^n$ with bounded first moment. Suppose that the support of $\mu$ is a connected bounded set with non-empty interior $K$. Then for every $x\in{\rm int}(\conv{K})$ there exists a halfspace $\mathcal{H}\sub\RR^n$ such that $b_\mu(\mathcal{H})=x$.
\end{lemma}
We remark that Lemma \ref{lem::exsitence} does not necessarily hold if the measure $\mu$ is not absolutely continuous. A simple example, in the plane, is the one-dimensional Lebesgue measure on the union of two opposite sides of a rectangle. Then, almost no point in the interval connecting the mid-points of those sides is a barycenter of some halfspace.

\begin{proof}[Proof of Lemma \ref{lem::exsitence}]
Assume without loss of generality that $x=0$.
Let 
$$K^{\circ}:=\{y\in{\mathbb R}^n\mid \forall x\in K: \iprod y x\leq 1\}$$
be the polar of $K$. Note that, by our assumptions on $K$, the polar set $K^\circ = (\conv{K})^\circ$, is a compact convex set, containing the origin in its interior. Hence,  the gauge function $||x||_{K^\circ}:=\inf\{r>0: x\in rK^\circ\}$ is well-defined. For every $0\neq y\in K^{\circ}$, consider the hyperplane
$$H_{y}:= \left\{x\in {\mathbb R}^n\mid \iprod y x= \left(1-\lVert y\rVert_{K^{\circ}}\right)\lVert y\rVert_{K^{\circ}}\right\} = H_{\theta, c},$$
where $\theta = \frac{y}{\|y\|}$ and $c = \frac{1}{\|y\|}\left(1-\lVert y\rVert_{K^{\circ}}\right)\lVert y\rVert_{K^{\circ}}$, and consider the map $f:K^{\circ}\to f(K^{\circ})$ defined by
\[
f(y)=\begin{cases}
b(K\cap H_{y}^{-}) & y\neq0\\
b\left(K\right) & y=0
\end{cases}.
\]
We claim that the map $f$ is continuous.
First note that the map
$$
h(y) := \begin{cases}
K\cap H_{y}^{-} & y\neq0\\
K & y=0
\end{cases}
$$
is continuous in the symmetric difference metric.
Indeed, since $K$ is bounded, one can choose a large enough ball $B_R$ such that $K \subset B_R$. 
Let $y \neq 0$. Since $\left(1-\lVert y\rVert_{K^{\circ}}\right)\lVert y\rVert_{K^{\circ}}$ is continuous with respect to $y$, and $y$ is perpendicular to $H_y$, then for $\| y - y' \|$ small enough from the absolute continuity of $\mu$ one gets that $\mu((B_R \cap H^-_y) \Delta (B_R \cap H^-_{y'})) < \varepsilon$, and hence $\mu((K \cap H^-_y) \Delta (K \cap H^-_{y'})) < \varepsilon$. 
For the continuity in the origin, one can consider for each $u \in S^{n-1}$ the point $y := \frac{u}{\|u\|_{K^o}}$, and the map $g_u(\delta) = \mu(H_{\delta y}^+ \cap K)$.
The map $u \mapsto \partial_+ g_u(0)$ is finite for each $u$ and it has a maximum from the compactness of $S^{n-1}$, and hence there exists a constant $C >0$ such that for $\delta$ small enough and for every $y$ that satisfies $\|y\|_{K^o} < \delta$ one has $\mu(H_y^+ \cap K) < C \delta$.
This together with the fact that $(H_y^- \cap K) \Delta K = H_y^+ \cap K$ shows continuity at the origin, and thus we showed that $h$ is continuous in the symmetric difference metric. As $f$ is the composition of the map $b$ with $h$
we are left with checking that the barycenter is continuous in the symmetric difference metric, which can be verified by using the following identity derived from \eqref{eq:LTE}.
\begin{align*} 
(b(K_1) - b(K_2)) \mu(K_1) \mu(K_2) =& \left(b(K_1 \cap K_2) \mu(K_1 \cap K_2) -b(K_2 \setminus K_1) \mu(K_1)\right)\mu(K_2 \setminus K_1) \\
&-\left(b(K_1 \cap K_2) \mu(K_1 \cap K_2)-b(K_1 \setminus K_2) \mu(K_2)\right) \mu(K_1 \setminus K_2).
\end{align*}
Therefore, $f$ is continuous, as claimed.

In addition, for every $y \neq 0$ one has $H_y \cap {\rm int}(\conv{K}) \neq \emptyset$ and hence $H_y \cap {\rm int}(K) \neq \emptyset$. By Lemma \ref{lem::uniqeness}
$f$ is a bijection, and hence a homeomorphism.
We deduce that $f(\partial K^{\circ})=\partial f(K^{\circ})$.
Note that since $f(\partial K^{\circ})$ is the set of barycenters $b\left(K\cap H^{-}\right)$
where the hyperplane $H$ goes through the origin, we get $b(K)\in{\rm int}(f(K^{\circ}))$.

We wish to prove that $0\in f(K^{\circ})$. To that end, it is sufficient
to prove that $[0,b(K)]\cap\partial f(K^{\circ})=\emptyset$, i.e., that $[0,b(K)]\cap f(\partial K^{\circ})=\emptyset$. Assume by contradiction
$\lambda b(K)\in[0,b(K)]\cap f(\partial K^{\circ})$.
 By the definition of $f$, there exists a $y\in \partial K^{\circ}$ (i.e., $\lVert y\rVert_{K^{\circ}}=1$) for which $\lambda b(K)=b(K\cap H^{-}_{y})$.
By \eqref{eq:LTE}, we have
\[
\lambda b(K)\mu(K\cap H^{-}_{y})+b(K\cap H^{+}_{y})\mu(K\cap H^{+}_{y})=b(K)\mu(K),
\]
and hence
\[
b(K\cap H^{+}_{y})\mu(K\cap H^{+}_{y})=b(K)\left[(1-\lambda)\mu(K\cap H^{-}_{y})+\mu(K\cap H^{+}_{y})\right].
\]
Since $\lambda<1$ it follows that $b(K\cap H^{+}_{y})=\alpha b(K)$ for
some $\alpha>0$. Since $b(K\cap H^{-}_{y})=\lambda b(K),\lambda>0$ and
$H_{y}$ passes through the origin, we get a contradiction.
\end{proof}

\begin{proof}[Proof of Theorem \ref{thm:main1}]

First one can assume that $\supp{\mu} = \RR^n$, otherwise set $\mu_\epsilon = \mu + \epsilon \mu_0$, where $\mu_0$ is the Gaussian measure. Note that for every $\delta>0$ there exists $\epsilon(\delta)$ small enough such that for every $A \subset \RR^n$, $|\mu_{\epsilon(\delta)}(A) - \mu(A)| < \delta$ and $\|b_{\mu_{\epsilon(\delta)}} (A) - b_\mu(A) \| < \delta$. Now for each $\delta$, assume that the theorem holds for $\mu_{\epsilon(\delta)}$. Choose a halfspace $\mathcal{H}_{\epsilon(\delta)}$ such that $b_{\mu_{\epsilon(\delta)}}(\mathcal{H}_{\epsilon(\delta)}) = x$ and $\mu_{\epsilon(\delta)}(\mathcal{H}_{\epsilon(\delta)}) = \sup\{\mu_{\epsilon(\delta)}(A)\,:\, b_{\mu_{\epsilon(\delta)}}(A)=x\}$. Now when taking $\delta \to 0$, we get that $\mathcal{H}_{\epsilon(\delta)} \xrightarrow{\delta \to 0} \mathcal{H}$ converges maybe after passing to a subsequence. Indeed, since $x \neq b_\mu(\RR^n)$, there exists some radius $R$ so that the distance of $\partial \mathcal{H}_{\epsilon(\delta)}$ from the origin is smaller than $R$, and the space of hyperplanes with distance from the origin smaller or equal than $R$ is compact. In addition, $x = b_{\mu_{\epsilon(\delta)}}(\mathcal{H}_{\epsilon(\delta)}) \to b_\mu(\mathcal{H})$ shows that indeed $b_\mu(\mathcal{H}) = x$. Finally for $A$ with $b_\mu(A) = x$, one can choose $A_{\epsilon(\delta)}$ with $b_{\mu_{\epsilon(\delta)}}(A_{\epsilon(\delta)}) = x$ and $\mu_{\epsilon(\delta)}(A_{\epsilon(\delta)}) \to \mu(A)$. One has $\mu(\mathcal{H}) = \lim \mu_{\epsilon(\delta)}(\mathcal{H}_{\epsilon(\delta)}) \geq \lim \mu_{\epsilon(\delta)}(A_{\epsilon(\delta)}) = \mu(A)$. This proves that indeed one is able to assume $\supp{\mu} = \RR^n$.
Now for every $R > 0$ such that $x \in B(R)$, the measure $\mu$ restricted to a ball of radius $R$ satisfy the conditions for Lemma \ref{lem::maximizer} and Lemma \ref{lem::exsitence} and hence there exists a hyperplane $\mathcal{H}_R$ such that $\mu_{R}(\mathcal{H}_{R}) = \sup\{\mu_{R}(A)\,:\, b_{\mu_{R}}(A)=x\}$. Note that since $\mu$ is bounded, for every $\epsilon>0$ there exists a large enough $R$ so that $\mu$ and $\mu_R$ agree up to $\epsilon$, so we can repeat the same arguments as before to complete the proof. 
\end{proof}

\section{A derived inequality}\label{sec:derived}

In this section we  discuss the derived inequality \eqref{eq:deriv_enough}. First, we  explain how it is related to Problem \ref{prob:main}. Second, in the case $n=1$, we establish a stronger inequality, namely Conjecture  \ref{conj:derivative},  for any log-concave distribution with mean $0$. Third, we prove \eqref{eq:deriv_enough} up to a universal constant, for all $n$. 

Let $X$ be a random variable which takes values in $[-1,1]$ such that for every $-1\leq t<1$, its tail $p_X(t):={\rm Pr}(X>t)$ is positive.
For any $-1\leq t< 1$ denote
$$m_X(t):=\ec{X}{X>t}.$$
Obviously, $t\leq m_X(t)\leq 1$ for every $-1\leq t<1$.
The function $m_X$ is clearly increasing in the interval $[-1,1)$; in particular, $m_X(t)\geq m_X(-1)=\EE[X]$ for every $-1\leq t<1$. 
Furthermore, for every $-1\leq t<1$ it holds that
\begin{equation}\label{eq:m}
m_X(t)=t+\frac{1}{p_X(t)}\int_t^1 p_X(x)\d x.
\end{equation} 

\subsection{Differentiating Problem \ref{prob:main}}
Let us  make the connection between \eqref{eq:diag} and \eqref{eq:deriv_enough}: 
The following lemma  shows that  inequality \eqref{eq:deriv_enough} is equivalent to the monotonicity of the function $\sqrt[n]{p_X(t)}/\sqrt{1-m_X(t)^2}$, from which \eqref{eq:diag} would immediately follow.

The lemma is formulated in a way that will allow us to prove \eqref{eq:diag} for all $t$ in the interval $(-1,1)$, although we are unable to establish \eqref{eq:deriv_enough} in the entire interval $(-1,1)$.

\begin{lemma}\label{lem:deriv}
Assume that $X$ is continuous with positive density in $(-1,1)$, and let $-1\leq\alpha<\beta\leq 1$. Suppose that $\sqrt[n]{p_X(\alpha)}\leq \sqrt{1-m_X(\alpha)^2}$ and \eqref{eq:deriv_enough} holds for every $\alpha<t<\beta$,  that is
\begin{equation*}
\left(m_X(t)-t\right)\frac{m_X(t)}{1-m_X(t)^2}<\frac{1}{n}.
\end{equation*}
Then \eqref{eq:diag} holds for every $\alpha< t<\beta$, that is
$$\sqrt[n]{p_X(t)}<\sqrt{1-m_X(t)^2}.$$ 
\end{lemma}

\begin{proof}
First note that $m_X(t)p_X(t)=t\, p_X(t)+\int_t^1 p_X(x)$ for any $-1<t<1$, by \eqref{eq:m}. Therefore,
for any $-1<t<1$ it holds that $m_X'(t)p_X(t)+m_X(t)p_X'(t)=tp_X'(t)$, i.e.,
\begin{equation}\label{eq:log_deriv}
\frac{m_X'(t)}{m_X(t)-t}=-\frac{p_X'(t)}{p_X(t)}.
\end{equation}For any $-1<t<1$, let
\begin{equation*}
f(t):=\ln\sqrt[n]{p_X(t)}-\ln\sqrt{1-m_X(t)^2}=\frac{1}{n}\ln p_X(t)-\frac{1}{2}\ln\left(1-m_X(t)^2\right).
\end{equation*} 
Then, for every $-1<t<1$, by using \eqref{eq:log_deriv},
\begin{equation*}
f'(t)=\frac{p_X'(t)}{n p_X(t)}+\frac{m_X'(t)m_X(t)}{1-m_X(t)^2}
=\frac{p_X'(t)}{p_X(t)}\left(\frac{1}{n}-\left(m_X(t)-t\right)\frac{m_X(t)}{1-m_X(t)^2}\right).
\end{equation*}
Since $p_X(t)>0$ and $p_X'(t)<0$ for any $-1<t<1$, it follows that $f'(t)>0$ for every $\alpha<t<\beta$. Hence, $f(t)>f(\alpha)\geq 0$ for every $\alpha< t<\beta$, and the result follows.
\end{proof}
 
\begin{remark}\label{rem:geometric}
A geometric way to derive \eqref{eq:deriv_enough} is  described as follows, without going into technical details: Denote $B=\left[-1,1\right]^{n}$. Let 
$H_{t}=\left\{ x\in\RR^{n}\mid\iprod {\mathbf{1}_n} x=\sqrt{n}t\right\} $,
$B_{t}=B\cap H_{t}^{+}$ and $b\left(B_{t}\right)=\gamma\mathbf{1}_n$.
Then 
\begin{align*}
\vol{B_{t}}= & \int_{t}^{1}\volco{B\cap H_{s}}\d s\\
b\left(B_{t}\right)\vol{B_{t}}= & \int_{t}^{1}b\left(B\cap H_{s}\right)\volco{B\cap H_{s}}\d s.
\end{align*}
By Theorem \ref{thm:main1} (and some symmetry considerations), it follows that for each $0<\gamma<1$ there exists $t\left(\gamma\right)$
such that $b\left(B_{t\left(\gamma\right)}\right)=\gamma\mathbf{1}_n$.
Note that for $\theta =\frac{1}{\sqrt{n}}\mathbf{1}_n$, 
we have $b_{\theta}\left(B_{t\left(\gamma\right)}\right)=\sqrt{n}\gamma$. Set 
\[
F\left(\gamma,t\right)=\sqrt{n}\gamma\vol{B_{t}}-\int_{t}^{1}b_\theta\left(B\cap H_{s}\right)\volco{B\cap H_{s}}\d s.
\]
Since $F\left(\gamma\right), t\left(\gamma\right))=0$, one can compute $dt/d\gamma$ (using the implicit function theorem) and show that the function 
$$
g:\gamma\mapsto\frac{\sqrt[n]{\vol{B_{t\left(\gamma\right)}}}}{2\sqrt{1-\left(\sqrt{n}\gamma\right)^{2}}}$$ 
is decreasing in $\gamma$ if and only  if  
\[
\gamma n^{2}\le\frac{\sqrt{n}\left(1-n\gamma^{2}\right)}{\sqrt{n}\gamma-b_\theta\left(B\cap H_{t\left(\gamma\right)}\right)}.
\]
Noting that $b_\theta\left(B\cap H_{t\left(\gamma\right)}\right)=t\left(\gamma\right)$
and recalling that $m(\gamma):=\sqrt{n}\gamma=b_\theta\left(B_{t\left(\gamma\right)}\right)$ one recovers \eqref{eq:deriv_enough}. Also note that since $g(0)=1$, \eqref{eq:deriv_enough} implies the validity of \eqref{eq:diag}.

\end{remark}   
    
\begin{remark}\label{rem:m_increasing}
Note that for every $-1\leq t<1$, the function 
$$m\mapsto (m-t)\frac{m}{1-m^2}$$ 
is increasing in the interval $[\max\{t,0\},1)$. Hence any upper bound on $\left(m_X(t)-t\right)\frac{m_X(t)}{1-m_X(t)^2}$ translates to an upper bound on $m_X(t)$ and vice-versa.

\end{remark}
\subsection{The case \texorpdfstring{$n=1$}{n=1}}\label{sec:n=1}

Recall that a plausible stronger form of \eqref{eq:deriv_enough} is given in Conjecture \ref{conj:derivative}. In this section we prove this conjecture for $n=1$. 

\begin{lemma}\label{lem:m_vs_t_half}
Suppose $X$ is a random variable with log-concave density $f:[-1,1]\to\RR$ such that $\EE[X]=0$. Then $m_X(t)\le \frac{1+t}{2}$ for all $t\in[-1,1]$.
\end{lemma} 

\begin{proof} 
We have $m_X(t) =\int_t^1 xf(x)\d x / \int_t^1 f(x)\d x$. Define 
$$F(t)=\int_t^1 xf(x)\d x -\frac{1+t}{2}\int_t^1 f(x)\d x.$$
We need to show that $F(t)\le 0$ for all $t\in[-1,1]$. We have $F(1)=0$ and also $F(-1)=0$ (since $\EE[X]=0$). 
By approximation, we may assume that $f$ is twice differentiable on $(-1,1)$. Note that  for any $-1<t<1$,
\begin{equation}\label{eq:Ftag}
F'(t) =-tf(t) -\frac{1}{2}\int_t^1f(x)\d x +\frac{1+t}{2}f(t)=\frac{1-t}{2}f(t)-\frac{1}{2}\int_t^1 f(x)\d x
\end{equation}
and 
\begin{equation}\label{eq:Ftagaim}
    F''(t)= -\frac{1}{2}f(t)+\frac{1-t}{2}f'(t)+\frac{1}{2}f(t)=\frac{1-t}{2}f'(t).
\end{equation}

Suppose that $f$ is monotone on $(-1,1)$. If $f$ is constant then the claim trivially holds. Otherwise, it follows by \eqref{eq:Ftag} that $F$ is monotone and non-constant, a contradiction to  the fact that $F(-1)=F(1)=0$. Hence, $f$ is not monotone on $(-1,1)$.

Since $\log f$ is concave, there exist $-1<p\le q<1$ such that $f$ is strictly increasing on $(-1,p)$, strictly decreasing on $(q,1)$ and constant on $[p,q]$. Therefore, on the one hand, by \eqref{eq:Ftag},  $F'\ge 0$ on $[p,1)$. As $F(1)=0$, it follows that  $F\le0$ on $[p,1)$. On the other hand, by \eqref{eq:Ftagaim},  $F''>0$ on $(-1,p)$ and so $F$ has no local maxima in $(-1,p)$. As $F(-1)=0$ and $F(p)\le 0$, it follows that $F\le 0$ on $(-1,p)$ as well.
\end{proof}

\begin{proposition}\label{prop:n=1}
Suppose $X$ is a random variable with log-concave density $f:[-1,1]\to\RR$ such that $\EE[X]=0$. Then
$$m_X(t)-t\le \frac{1-m_X(t)^2}{2m_X(t)}.$$
Moreover, the uniform distribution is the extremizing distribution in this case.
\end{proposition} 

\begin{proof}
Let $U$ be a uniformly distributed random variable on $[-1,1]$. 
Note that $m_U(t)=(1+t)/2$ and hence 
$$m_U(t)-t=\frac{1-t}{2}\le\frac{1-((1+t)/2)^2}{2(1+t)/2}=\frac{1-m_U(t)^2}{2m_U(t)}.$$
By Lemma \ref{lem:m_vs_t_half}, we also have $0\leq m_X(t)\leq m_U(t)$ for all $t\in[-1,1]$. Moreover, since the function $g(x)=\frac{1-x^2}{2x}$ is decreasing on $(0,1]$, we have
$$ m_X(t)-t\le m_U(t)-t\le g(m_U(t))\le g(m_X(t))=\frac{1-m_X(t)^2}{2m_X(t)},$$
as claimed.
\end{proof}
Notice that the proofs of the above lemma and proposition remain valid under weaker assumptions. For example, it is also valid if $f$ is unimodal, i.e., for some $p$ in $(-1,1)$, $f$ is increasing  on $(-1,p)$ and decreasing on $(p,1)$.

Also note that, contrary to Proposition \ref{prop:n=1}, Conjecture \ref{conj:derivative} fails for $n\ge2$ if one only assumes that the density of $X$ is some log-concave function supported on $[-1,1]$ and that $X$ has the same expectation and variance as that of the average of $n$ uniform random variables on $[-1,1]$. For example, one can verify that this is case for the appropriate truncated normal distribution. In fact, for this distribution, also \eqref{eq:deriv_enough} is violated for $n\ge3$.  
It is not clear to us, what is the ``right" family of distribution for which Conjecture \ref{conj:derivative} or \eqref{eq:deriv_enough} might hold. One possible family is that of $\frac{1}{n-1}$-concave distributions. 

\subsection{Proof of \eqref{eq:deriv_enough} up to a universal constant}\label{sec:derived_C}

Let $U_1,\ldots,U_n$ be independent uniform random variables on $[-1,1]$ and let $X:=\frac{1}{n}\sum_{i=1}^n U_i$.
In this section we prove the following relaxed version of \eqref{eq:deriv_enough}.

\begin{proposition}\label{prop:derived_C}
There is a universal constant $C$ such that for every $-1 < t< 1$,
\begin{equation*}
 \left(m_X(t)-t\right)\frac{m_X(t)}{1-m_X(t)^2}<\frac{C}{n}.
\end{equation*}
\end{proposition}

In preparation for the proof of Proposition \ref{prop:derived_C}, we gather the following simple facts whose proofs are rather standard. For that reason, and to make the exposition clearer, we give their proofs in Appendix \ref{sec:appendix}, for completeness.

\begin{obs}\label{obs:one minus}
For every $-1< t< 1$,
\begin{equation*}\label{eq:minus}
m_X(-t)=\frac{p_X(t)}{1-p_X(t)}m_X(t).
\end{equation*}
\end{obs}

\begin{claim}\label{claim:var}
For every $0\leq t< 1$,
\begin{equation*}\label{eq:var}
m_X(t)\leq t+\sqrt{\EE[X^2]}.
\end{equation*}
\end{claim}

\begin{claim}\label{claim:small_exact}
The function $p_X$ is $\frac{1}{n}$-concave in the interval $[-1,1]$.
Consequently, for every $-1\leq t< 1$,
\begin{equation}\label{eq:small_exact}
p_X(t)\leq\frac{(n(1-t)/2)^n}{n!}<(e(1-t)/2)^n.
\end{equation}
Furthermore, for every $0<t\leq 1$,
\begin{equation}\label{eq:chernoff}
p_X(t)< e^{-3nt^2/2}.
\end{equation}
\end{claim}

We remark that the above facts, as well as the following lemma, hold under weaker assumptions on $X$. For simplicity, we refrain from formulating them in the most general setting.\\

\begin{lemma}\label{lem:e_half_plain}
For every $0<t<1$ such that $\sqrt[n]{p_X(t)}<n/(n+1)$,
\begin{equation}\label{eq:m_concave_bound}
\frac{m_X(t)}{t}\leq 1+\frac{\sqrt[n]{p_X(t)}}{n-(n+1)\sqrt[n]{p_X(t)}}.
\end{equation}
Consequently, if $\frac{1}{\sqrt[n]{p_X(t)}}>1+\frac{2}{(1-t^2)n}$, then
\begin{equation}\label{eq:derived_concave_bound} 
\left(m_X(t)-t\right)\frac{m_X(t)}{1-m_X(t)^2}<\frac{t^2}{\left(\frac{1}{\sqrt[n]{p_X(t)}}-1\right)(1-t^2)n-2}.
\end{equation}
\end{lemma}

\begin{proof}
For simplicity, denote $\tau:=\sqrt[n]{p_X(t)}$.
For every $t \leq u\leq 1$, since $p_X$ is $\frac{1}{n}$-concave in the interval $[-1,1]$, by Claim \ref{claim:small_exact},
$$
\left(1-\frac{t}{u}\right)\sqrt[n]{p_X\left(0\right)}+\frac{t}{u}\sqrt[n]{p_X\left(u\right)}\leq\sqrt[n]{p_X\left(\frac{t}{u} u\right)}=\sqrt[n]{p_X\left(t\right)}$$
and hence,
\begin{equation*}
\sqrt[n]{p_X\left(u\right)}\leq\frac{u}{t}\sqrt[n]{p_X\left(t\right)}+\left(1-\frac{u}{t}\right)\sqrt[n]{p_X(0)}=\frac{u}{t}\tau+\left(1-\frac{u}{t}\right)\frac{1}{\sqrt[n]{2}}.
\end{equation*}
(In particular, $\frac{1}{t}\tau+\left(1-\frac{1}{t}\right)\frac{1}{\sqrt[n]{2}}\geq\sqrt[n]{p_X(1)}=0$.)
Therefore, 
\begin{align*}
\int_t^1 p_X(u)\d u&\leq \int_t^1\left(\frac{u}{t}\tau+\left(1-\frac{u}{t}\right)\frac{1}{\sqrt[n]{2}}\right)^n \d u\\
&=\frac{t}{(n+1)\left(\frac{1}{\sqrt[n]{2}}-\tau\right)}\left(\tau^{n+1}-\left(\frac{1}{t}\tau+\left(1-\frac{1}{t}\right)\frac{1}{\sqrt[n]{2}}\right)^{n+1}\right)\\
&\leq\frac{t\tau}{(n+1)\left(\frac{1}{\sqrt[n]{2}}-\tau\right)}\tau^n\leq\frac{t\tau}{(n+1)\left(\frac{n}{n+1}-\tau\right)}\tau^n\\
&=\frac{t\sqrt[n]{p_X(t)}}{n-(n+1)\sqrt[n]{p_X(t)}}p_X(t),
\end{align*}
and \eqref{eq:m_concave_bound} follows by \eqref{eq:m}.
We proceed to prove \eqref{eq:derived_concave_bound}. Denote
$$u(t):=t\left(1+\frac{\sqrt[n]{p_X(t)}}{n-(n+1)\sqrt[n]{p_X(t)}}\right).$$
By \eqref{eq:m_concave_bound},
$$
\left(m_X(t)-t\right)\frac{m_X(t)}{1-m_X(t)^2}\leq\left(u(t)-t\right)\frac{u(t)}{1-u(t)^2},
$$
and a straightforward computation shows that
\begin{align*}
\left(u(t)-t\right)\frac{u(t)}{1-u(t)^2}&=\frac{t^2}
{\left(\frac{1}{\sqrt[n]{p_X(t)}}-1\right)(1-t^2)n-2+\frac{\sqrt[n]{p_X(t)}}{\left(1-\sqrt[n]{p_X(t)}\right)n}}\\
&\leq\frac{t^2}
{\left(\frac{1}{\sqrt[n]{p_X(t)}}-1\right)(1-t^2)n-2}.
\qedhere\end{align*}
\end{proof}

The proof of Proposition \ref{prop:derived_C}, as well as the proof of Theorem \ref{thm:cube_main} in the next section, are based on the following Lemma.
\begin{lemma}\label{lem:derived_C}
Let $\alpha,\beta,\gamma$ be positive real numbers such that $\alpha,\beta\leq \sqrt{3n}$ and $\gamma\leq 1$.

\begin{enumerate}
\item If $\alpha>\sqrt{2}$ then for every $-1< t \leq-\frac{\alpha}{\sqrt{3n}}$,
$$\left(m_X(t)-t\right)\frac{m_X(t)}{1-m_X(t)^2}< h_1(\alpha,n):=\frac{\alpha^6}{(\alpha^2-2)^2\left(e^{\alpha^2/2}-1\right)}\cdot\frac{1}{3n-1}.$$

\item For every $-\frac{\alpha}{\sqrt{3n}}< t \leq 0$,
$$\left(m_X(t)-t\right)\frac{m_X(t)}{1-m_X(t)^2}<h_2(\alpha,n):=\frac{1+\alpha}{3n-1}.$$

\item If $\beta<\sqrt{3n}-1$ then for every $0< t<\frac{\beta}{\sqrt{3n}}$,
$$\left(m_X(t)-t\right)\frac{m_X(t)}{1-m_X(t)^2}< h_3(\beta,n):=\frac{1+\beta}{3n-(1+\beta)^2}.$$

\item If $\frac{16}{\beta^2}<(4+3\gamma^2)(1-\gamma^2)$ then for every $\frac{\beta}{\sqrt{3n}}\leq t\leq\gamma$,
$$\left(m_X(t)-t\right)\frac{m_X(t)}{1-m_X(t)^2}<h_4(\beta, \gamma,n):=\frac{8/3}{\left((4+3\gamma^2)(1-\gamma^2)-\frac{16}{\beta^2}\right)n}.$$

\item If $\left(\left(1+\frac{2}{e}\right)-\left(1-\frac{2}{e}\right)\frac{1}{\gamma^2}\right)n>\frac{2}{\gamma^2}$ then for every $\gamma\leq t<1$,
$$\left(m_X(t)-t\right)\frac{m_X(t)}{1-m_X(t)^2}<h_5(\gamma,n):=\frac{1}{\left(\left(1+\frac{2}{e}\right)-\left(1-\frac{2}{e}\right)\frac{1}{\gamma^2}\right)n-\frac{2}{\gamma^2}}.$$
\end{enumerate}
\end{lemma}

\begin{proof}
Note first that if $-1\leq t\leq 0$ then by Claim \ref{claim:var}, 
\begin{equation}\label{eq:tleq0}
m_X(t)\leq m_X(0)\leq \sqrt{\EE[X^2]}=\frac{1}{\sqrt{3n}}.
\end{equation}

Suppose that $-1< t\leq -\frac{\alpha}{\sqrt{3n}}$.
By \eqref{eq:chernoff}, 
$$
\sqrt[n]{p_X(-t)}\leq e^{-\frac{3}{2}t^2}\leq e^{-\frac{\alpha^2}{2n}}< \frac{1}{1+\frac{\alpha^2}{2n}}.
$$
Therefore, by \eqref{eq:m_concave_bound}, 
\begin{equation*}
\frac{m_X(-t)}{-t}< 1+\frac{\frac{1}{1+\frac{\alpha^2}{2n}}}{n-(n+1)\frac{1}{1+\frac{\alpha^2}{2n}}}=\frac{\alpha^2}{\alpha^2-2}
\end{equation*}
and hence, by Observation \ref{obs:one minus} and \eqref{eq:chernoff},
$$
m_X(t)=\frac{m_X(-t)}{\frac{1}{p_X(-t)}-1}<\frac{\alpha^2}{\alpha^2-2}\cdot\frac{-t}{ e^{3nt^2/2}-1}.
$$
In particular,
$$
m_X(t)<\frac{\alpha^2}{\alpha^2-2}\cdot\frac{-t}{ e^{\alpha^2/2}-1}< \frac{2}{\alpha^2-2}(-t).
$$
Therefore, by \eqref{eq:tleq0}, and since the function $x\mapsto\frac{x}{e^x-1}$ is decreasing in the interval $(0,\infty)$,
\begin{align*}
\left(m_X(t)-t\right)&\frac{m_X(t)}{1-m_X(t)^2}<\left(\frac{2}{\alpha^2-2}(-t)-t\right) \frac{\frac{\alpha^2}{\alpha^2-2}\cdot\frac{-t}{ e^{3nt^2/2}-1}}{1-\frac{1}{3n}}\\
=&\frac{2\alpha^4}{(\alpha^2-2)^2}\cdot\frac{3n\,t^2/2}{e^{3n\,t^2/2}-1}\cdot\frac{1}{3n-1}\leq\frac{2\alpha^4}{(\alpha^2-2)^2}\cdot\frac{\alpha^2/2}{e^{\alpha^2/2}-1}\cdot\frac{1}{3n-1}\\
=&\frac{\alpha^6}{(\alpha^2-2)^2\left(e^{\alpha^2/2}-1\right)}\cdot\frac{1}{3n-1}.
\end{align*}

If $-\frac{\alpha}{\sqrt{3n}}< t\leq 0$ then by \eqref{eq:tleq0},
\begin{equation*}
\left(m_X(t)-t\right)\frac{m_X(t)}{1-m_X(t)^2}\leq \left(\frac{1}{\sqrt{3n}}-t\right)\frac{\frac{1}{\sqrt{3n}}}{1-\frac{1}{3n}}<\frac{1+\alpha}{\sqrt{3n}}\cdot\frac{\frac{1}{\sqrt{3n}}}{1-\frac{1}{3n}}=\frac{1+\alpha}{3n-1}.
\end{equation*}

If $0< t<\frac{\beta}{\sqrt{3n}}$, then by Claim \ref{claim:var}, 
$m_X(t)\leq \sqrt{\EE[X^2]}+t=\frac{1}{\sqrt{3n}}+t<\frac{1+\beta}{\sqrt{3n}}$ 
and hence,
\begin{equation*}
\left(m_X(t)-t\right)\frac{m_X(t)}{1-m_X(t)^2}<\frac{1}{\sqrt{3n}}\cdot\frac{\frac{1+\beta}{\sqrt{3n}}}{1-\frac{(1+\beta)^2}{3n}}=\frac{1+\beta}{3n-(1+\beta)^2}.
\end{equation*}

Next, suppose that $\frac{\beta}{\sqrt{3n}}\leq t\leq\gamma$.
Then, by \eqref{eq:chernoff}, 
\begin{align*}
\frac
{\left(\frac{1}{\sqrt[n]{p_X(t)}}-1\right)(1-t^2)n-2}{t^2}&>\frac{n\left(e^{3t^2/2}-1\right)(1-t^2)-2}{t^2}\\
&\geq
\frac{n\left(\frac{3}{2}t^2+\frac{1}{2}\left(\frac{3}{2}\right)^2t^4\right)(1-t^2)-2}{t^2}\\
&=\frac{(4+3t^2)(1-t^2)n-\frac{16}{3t^2}}{8/3}\geq\frac{\left((4+3\gamma^2)(1-\gamma^2)-\frac{16}{\beta^2}\right)n}{8/3}>0,
\end{align*}
and the desired bound follows by \eqref{eq:derived_concave_bound}.

Finally, suppose that $\gamma\leq t<1$. Then, by \eqref{eq:small_exact}, 
\begin{align*}
\frac
{\left(\frac{1}{\sqrt[n]{p_X(t)}}-1\right)(1-t^2)n-2}{t^2}&\geq\frac{n\left(\frac{2}{e(1-t)}-1\right)(1-t^2)-2}{t^2}=\frac{\frac{2}{e}n(1+t)-n(1-t^2)-2}{t^2}\\
&>\frac{\frac{2}{e}n(1+t^2)-n(1-t^2)-2}{t^2}\\
&=\left(\left(1+\frac{2}{e}\right)-\left(1-\frac{2}{e}\right)\frac{1}{t^2}\right)n-\frac{2}{t^2}\\
&\geq\left(\left(1+\frac{2}{e}\right)-\left(1-\frac{2}{e}\right)\frac{1}{\gamma^2}\right)n-\frac{2}{\gamma^2}>0,
\end{align*}
and the desired bound follows by \eqref{eq:derived_concave_bound}.
\end{proof}
Proposition \ref{prop:derived_C} easily follows from 
Lemma \ref{lem:derived_C}.
\begin{proof}[Proof of Proposition \ref{prop:derived_C}]
Invoke Lemma \ref{lem:derived_C}, for instance with $\alpha=2$, $\beta=3$ and $\gamma=1/2$. 
This yields that for every $n\geq 12$,
$$\left(m_X(t)-t\right)\frac{m_X(t)}{1-m_X(t)^2}
<\begin{cases}
h_1(2,n)=\frac{16}{\left(e^2-1\right)(3n-1)} & -1< t\leq -\frac{2}{\sqrt{3n}}\\
h_2(2,n)=\frac{3}{3n-1} & -\frac{2}{\sqrt{3n}}<t\leq 0\\
h_3(3,n)=\frac{4}{3n-16} & 0< t< \frac{3}{\sqrt{3n}}\\
h_4(3,\frac{1}{2},n)=\frac{384}{257\,n} & \frac{3}{\sqrt{3n}}\leq t<\frac{1}{2}\\
h_5(\frac{1}{2},n)=\frac{1}{\left(\frac{10}{e}-3\right)n-8} & \frac{1}{2}\leq t<1
\end{cases}.
$$
Hence, for every $n\geq 12$ and every $-1< t<1$,
$$
\left(m_X(t)-t\right)\frac{m_X(t)}{1-m_X(t)^2}<\frac{3}{2n-15}
$$
and Proposition \ref{prop:derived_C} readily follows. 
\end{proof}

Note that there is no choice of parameters (specifically, $\beta$) in Lemma \ref{lem:derived_C} that would show that \eqref{eq:deriv_enough} holds for every $-1< t<1$. 

\section{Isobarycentric inequalities for the cube}\label{sec:iso_cube}

\subsection{The diagonal case}\label{subsec:diag}
In this section we prove Theorem \ref{thm:cube_main}.
\begin{proposition}\label{prop:cube_reduced}
There is a positive integer $n_0$ such that \eqref{eq:diag} holds for every $n\geq n_0$ and every $-1 < t< 1$, i.e., 
\begin{equation*}
\sqrt[n]{p_X(t)} <   \sqrt{1-m_X(t)^2} \end{equation*}
where $X=\frac{1}{n}\sum_{i=1}^n U_i$, and $U_1,\ldots,U_n$ are independent uniform random variables on $[-1,1]$.
\end{proposition}

\begin{proof}
First note that
$$
n\,h_1\left(2-\frac{1}{n},n\right)\xrightarrow[n\to\infty]{}\frac{16}{3\left(e^2-1\right)}<1, \quad n\,h_3(\sqrt{3},n)\xrightarrow[n\to\infty]{}\frac{1+\sqrt{3}}{3}<1,
$$
and
$$
n\,h_5\left(\frac{2}{3},n\right)\xrightarrow[n\to\infty]{}\frac{1}{\left(1+\frac{2}{e}\right)-\left(1-\frac{2}{e}\right)\frac{9}{4}}<1,
$$
where $h_1,h_3,h_5$ are as in Lemma \ref{lem:derived_C}.
Hence, there is a positive integer $n_0$ such that for every $n>n_0$,
\begin{equation}\label{eq:4ineq}
h_1\left(2-\frac{1}{n},n\right)<\frac{1}{n},\quad h_3(\sqrt{3},n)<\frac{1}{n},\quad
\frac{3}{1+\frac{64}{n}}>\frac{4+2\sqrt{3}}{3}, \quad 
h_5\left(\frac{2}{3},n\right)<\frac{1}{n}.
\end{equation}

Fix $n\ge n_0$. The first three parts of 
Lemma \ref{lem:derived_C}, with $\alpha:=2-\frac{1}{n}$ and $\beta:=\sqrt{3}$, combined with the first two inequalities in \eqref{eq:4ineq} and the fact that  $h_2(2-\frac{1}{n},n)=\frac{1}{n}$, yield that \eqref{eq:deriv_enough} holds for every $-1<t<\frac{1}{\sqrt{n}}$. 
Since $\sqrt[n]{p_X(-1)}=1=\sqrt{1-m_X(-1)^2}$, it follows from Lemma \ref{lem:deriv} that \eqref{eq:diag} holds for every $-1<t<\frac{1}{\sqrt{n}}$.

For every $\frac{1}{\sqrt{n}}\leq t\leq\frac{8}{\sqrt{3n}}$,
by \eqref{eq:chernoff},  \eqref{eq:var} and the third inequality in \eqref{eq:4ineq},
\begin{align*}
\left(\sqrt[n]{p_X(t)}\right)^2+m_X(t)^2&< e^{-3t^2}+\left(t+\frac{1}{\sqrt{3n}}\right)^2< \frac{1}{1+3t^2}+\left(t+\frac{1}{\sqrt{3}}t\right)^2\\
&=1-t^2\left(\frac{3}{1+3t^2}-\frac{4+2\sqrt{3}}{3}\right)\\
&\leq 1-t^2\left(\frac{3}{1+\frac{64}{n}}-\frac{4+2\sqrt{3}}{3}\right)<1,
\end{align*}
i.e., $\sqrt[n]{p_X(t)}<\sqrt{1-m_X(t)^2}$.

In particular, \eqref{eq:diag} holds for $t=8/\sqrt{3n}$.
The last two parts of 
Lemma \ref{lem:derived_C}, with $\beta:=8$ and $\gamma:=2/3$, yield, by using the last inequality in \eqref{eq:4ineq} and since $h_4(8, \frac{2}{3},n):=\frac{288}{293\,n}<\frac{1}{n}$, that \eqref{eq:deriv_enough} holds for every $\frac{8}{\sqrt{3n}}<t<1$.
Hence, it follows from Lemma \ref{lem:deriv} that \eqref{eq:diag} holds for every $\frac{8}{\sqrt{3n}}<t<1$, which concludes the proof.
\end{proof}

\begin{proof}[Proof of Theorem \ref{thm:cube_main}]
Let $n_0$ be as in Proposition \ref{prop:cube_reduced}. Then Theorem \ref{thm:main1} and Proposition \ref{prop:cube_reduced}, combined with some symmetry considerations, yield that for every $n>n_0$, every non-zero $-1<m<1$ and each $K \subset [-1,1]^n$ such that $b(K) = m\mathbf{1}_n$, 
one has
$$\sqrt[n]{\vol{K}} < 2 \sqrt{1-m^2}.$$
The validity of the theorem for all $n$ follows by a simple tensorization argument. 
Indeed, for any positive integer $n$ there is a positive integer $\ell$ such that $\ell\,n>n_0$. Then for every non-zero $-1<m<1$ and each $K \subset [-1,1]^n$ such that $b(K) = m\mathbf{1}_n$, 
we have $K^{\ell}\subset [-1,1]^{\ell n}$, 
$b\left(K^{\ell}\right)=m\mathbf{1}_{\ell n}$
and hence
$\sqrt[n]{\vol{K}}=\sqrt[\ell n]{\vol{K^{\ell}}}<2\sqrt{1-m^2}$.
\end{proof}

\subsection{A general bound}

It is natural to compare the conjectured upper bound in Problem  \ref{prob:main} to that obtained by $x$-symmetric subsets of $[-1,1]^n$, that is subsets $L\sub[-1,1]^n$ such that $L-x=x-L$. For these subsets, one clearly obtains 
\begin{equation*}
\vol{L}\le 2^n\prod_{i=1}^n (1-|x_i|),
\end{equation*}
where the maximum is attained at $L=[-1,1]^n\cap(2x-[-1,1]^n)$. Hence, by the Milman-Pajor inequality \cite{MP} (also see \cite{HTSV} for recent improvements of the inequality), for every convex body $K\sub\RR^n$ such that $b(K)=x$, we have 
\begin{equation}\label{eq:MP}
\vol{K}\le 2^n\,\vol{K\cap(2x -K)}\le 4^n\prod_{i=1}^n (1-|x_i|).    
\end{equation}
It is conjectured that the extremal body in  the  Milman-Pajor inequality is the simplex, for which the constant obtained is of the order of $(e/2)^n$ (instead of $2^n$). This stronger version of the Milman-Pajor inequality would imply that \eqref{eq:MP} holds with a constant of the order of $e^n$ instead of $4^n$. We can prove the latter directly, by an argument used in the proof of \cite{HP}*{Theorem 1.4}, which employs  Gr\"{u}nbaum's inequality \cite{Grunbaum60}.

\begin{proposition}\label{prop:Greentree}
Let $K\subseteq[-1,1]^n$ be a measurable set with positive volume and suppose that $b(K)=(x_1,\dots,x_n)$. Then,
\begin{equation}\label{eq:Greentree}
   \vol{K}<e^n \prod_{i=1}^n\left(1-|x_i|\right).
\end{equation}
\end{proposition}

We remark that the constant $e$ in \eqref{eq:Greentree} may not be replaced by a smaller constant. 
Indeed, for any $1-\frac{1}{n}<t<1$, consider the simplex $K_{n,t}:=\left\{y\in[-1,1]^n\mid\iprod{\mathbf{1}_n}y\geq n\,t\right\}$. Then
$b\left(K_{t,n}\right)=s\mathbf{1}_n$ where $s=(n\,t+1)/{(n+1)}>t$, $\vol{K_{n,t}}=n^n(1-t)^n/n!$ and hence
$$
\frac{\sqrt[n]{\vol{K_{n,t}}}}{1-s}>\frac{\sqrt[n]{\vol{K_{n,t}}}}{1-t}=\frac{n}{\sqrt[n]{n!}}\xrightarrow[n\to\infty]{}e.
$$

\begin{proof}[Proof of Proposition \ref{prop:Greentree}]
With no loss of generality we may assume that $x_1,\ldots,x_n$ are all non-negative. We may also assume that $K$ is convex, by Theorem \ref{thm:main1}.
By Gr\"unbaum's inequality \cite{Grunbaum60}, for $\theta:=\left(\frac{1}{1-x_1},\ldots,\frac{1}{1-x_n}\right)$, we have
\begin{align*}
\left(\frac{n}{n+1}\right)^{n}\vol K  &\leq \vol{\left\{ y\in K\mid \iprod{\theta}{y-b(K)}\geq 0\right\} }\\
&\leq\vol{\left\{ y\in(-\infty,1]^n\mid \iprod{\theta}{y-b(K)}\geq 0\right\} }\\
 &=\vol{\left\{ y\in(-\infty,1]^n\mid \iprod{\theta}{\mathbf{1}_n-y}\leq \iprod{\theta}{\mathbf{1}_n-b(K)}=n \right\} }\\
 &=\vol{\left\{ z\in[0,\infty)^n\mid \iprod{\theta}{z}\leq n \right\} }= \frac{n^n}{n!}\prod_{i=1}^n\left(1-x_i\right)
\end{align*}
and hence, 
\begin{equation*}
\vol K\leq\frac{(n+1)^n}{n!}\prod_{i=1}^n\left(1-x_i\right)=\prod_{m=1}^n\left(1+\frac{1}{m}\right)^m\prod_{i=1}^n\left(1-x_i\right)<e^n\prod_{i=1}^n\left(1-x_i\right).
\qedhere\end{equation*}
\end{proof}

As mentioned, the above proof is based on the proof of \cite{HP}*{Theorem 1.4}, which is a Log-Minkowski inequality for a centered simplex and a centered convex body. We note that this inequality may be formulated as an isobarycenteric inequality for a simplex, say 
$S=\{x\in[0,\infty)^n\mid\iprod{\mathbf{1}_n}x\leq 1\}$, as follows: given a convex body $K\subseteq S$ such that $b(K)=(x_1,\ldots,x_n)$, one has
\begin{equation}\label{eq:simplex_isobarycentric}
\vol{K} \leq \frac{(n+1)^n}{n!}\left[ \left(1-\sum_{i=1}^n x_i\right) \prod_{i=1}^n x_i \right]^\frac{n}{n+1}.    
\end{equation}

We remark that an argument similiar to the proof of Proposition \ref{prop:Greentree} yields the following slightly stronger inequality:
$$
\vol{K} \leq  \frac{(n+1)^n}{n!}\cdot\frac{\left(1-\sum_{i=1}^n x_i\right) \prod_{i=1}^n x_i}{\max\left\{x_1,\ldots,x_n,1-\sum_{i=1}^n x_i\right\}}.
$$

\section{Discussion and open questions}\label{sec:discussion}
 
\subsection{}
It is not clear whether a complete solution to Problem \ref{prob:main} can be obtained by a reduction to the diagonal case, i.e., Theorem \ref{thm:cube_main}. 
A different approach to resolve Problem \ref{prob:main} in the general case is the following. Let $x=(x_1,\dots,x_n)\in[-1,1]^n$ and let $K_x$ be the unique intersection of $[-1,1]^n$ with a halfspace such that $b(K_x)=x$. By Theorem \ref{thm:main1}, such an intersection exists and it has maximal volume among all subsets of $[-1,1]^n$ whose barycenter is $x$. Numerical evidence suggests that $$\frac{\vol{K_x}}{\prod_{i=1}^n \sqrt{1-x_i^2}}$$ increases as $|x_i|$ decreases for each coordinate separately, which would clearly yield an affirmative answer to the problem.

\subsection{}
A natural question which generalizes Problem \ref{prob:main} is the following: Given a convex body $L\sub\RR^n$ and a point $x\in L$, what can be said about the maximal volume of a subset $K$ of $L$ which is centered at $x$? This question can also be extended to the realm of measures.
While Theorem \ref{thm:main1} characterizes the shape of $K$, it does not provide a bound on its volume.
Also see \eqref{eq:simplex_isobarycentric} for the case where $L$ is a simplex. 

In a dual direction, given a convex body $L\sub\RR^n$ and a point $x\in L$, what is the minimal volume of a set $K$ containing $L$, which is centered at $x$\,?
In \cite{HMC2016}*{Lemma 2.2} a special case of this problem was resolved. For $L = \conv{\{0,u_1,\ldots,u_n\}}$ and $K\supset L$ with $b(K)=0$, one has
$$ \vol{K} \geq \frac{n+1}{n!} | \det(u_1,\ldots,u_n) | $$
with equality only for $K = \conv{\{u_1,\ldots,u_n,-(u_1+\cdots+u_n)\}}$.

\subsection{}
Let $Y_1,\dots Y_n$ be independent log-concave random variables on $\RR$ whose density is even. It is a known fact that the ``peakedness" of $X_n=\frac{1}{n}\sum_{i=1}^n Y_i$ is increasing with $n$, see e.g., \cites{P, SW} and references therein. This means that $p_{X_n}(t)=\PP(X_n>t)$ increases with $n$ for $t<0$ and decreases for $t>0$. (In other words, $\PP(|X_n|<t)$ increases with $n$.) In particular, if $X_n$ is the average of $n$ independent uniformly distributed random variables on $[-1,1]$, as in \eqref{eq:diag}, $p_{X_n}(t)$ is decreasing with $n$ for $t>0$. However, numerical evidence suggests that both $\sqrt[n]{p_{X_n}(t)}$ and $m_X(t)$ are increasing with $n$ for any $-1<t<1$. It would be interesting to prove these monotonicity properties of $\sqrt[n]{p_{X_n}}$ and $m_{X_n}$, although this would not provide an alternative proof for \eqref{eq:diag}.

\subsection{}
Our interest in \eqref{eq:deriv_enough} originated from Problem \ref{prob:main}. However, as mentioned in Remark \ref{rem:m_increasing}, this inequality reflects an upper bound on the truncated mean $m_X(t)$ where $X$ is the average of $n$ uniformly distributed random variables on $[-1,1]$. It is intriguing to find a more natural bound for $m_X(t)$ in this case, and to study $m_X(t)$ for more general distributions. In particular,  it would be interesting to obtain bounds on the truncated mean, depending on various parameters such as the variance of $X$, its degree of concavity, entropy etc. 
\appendix

\section{}\label{sec:appendix}

\begin{proof}[Proof of Observation \ref{obs:one minus}]
Note that
$$
\ec{X}{X<t}=\ec{-X}{-X<t}=-\ec{X}{X>-t}=-m_X(-t).
$$
Therefore, 
\begin{align*}
0&=\EE[X]={\rm Pr}(X>t)\ec{X}{X>t}+{\rm Pr}(X<t)\ec{X}{X<t}\\
&=p_X(t)m_X(t)-\left(1-p_X(t)\right)m_X(-t),
\end{align*}
which concludes the proof.
\end{proof}

\begin{proof}[Proof of Claim \ref{claim:var}]
For every $-1\leq t\leq 1$,
$$
\int_t^1 p_X(x)\d x=\frac{1}{2^n}{\rm vol}\left(\left\{(u,x)\in[-1,1]^n\times[-1,1]\,\bigg|\, \frac{1}{n}\iprod {\mathbf{1}_n} u >x>t\right\}\right).
$$
Hence, by the Brunn-Minkowski inequality, the function $t\mapsto\int_t^1 p_X(x)\d x$ is log-concave in the interval $[-1,1]$.
Therefore, its logarithmic derivative 
$t\mapsto\frac{-p_X(t)}{\int_t^1 p_X(x)\d x}$ is decreasing in the interval $[-1,1)$.
Hence, the function $t\mapsto \frac{1}{p_X(t)}\int_t^1 p_X(x)\d x$ is decreasing in the interval $[-1,1)$.
In particular, for every $0\leq t<1$, by \eqref{eq:m},
\begin{equation*}
m_X(t)-t=\frac{1}{p_X(t)}\int_t^1 p_X(x)\d x\leq \frac{1}{p_X(0)}\int_0^1 p_X(x)\d x=m_X(0)-0.
\end{equation*}
The claim follows since by the Cauchy-Schwartz inequality,
\begin{equation*}
m_X(0)=\ec{X}{X>0}\leq\sqrt{\ec{X^2}{X>0}}=\sqrt{\EE[X^2]}. 
\qedhere\end{equation*}
\end{proof}

\begin{proof}[Proof of Claim \ref{claim:small_exact}]
For every $-1\leq t\leq 1$,
$$
p_X(t)=\frac{1}{2^n}{\rm vol}\left(\left\{u\in[-1,1]^n\,\bigg|\, \frac{1}{n}\iprod {\mathbf{1}_n} u >t\right\}\right).
$$
Hence, $p_X$ is $\frac{1}{n}$-concave in the interval $[-1,1]$, by the Brunn-Minkowski inequality. 
In particular, for every $-1\leq t<x\leq 1$, 
\begin{align*}
\frac{1-x}{1-t}\sqrt[n]{p_X\left(t\right)}&=\frac{1-x}{1-t}\sqrt[n]{p_X\left(t\right)}+\frac{x-t}{1-t}\sqrt[n]{p_X\left(1\right)}\\
&\leq\sqrt[n]{p_X\left(\frac{1-x}{1-t}t+\frac{x-t}{1-t} 1\right)}=\sqrt[n]{p_X\left(x\right)}
\end{align*}
and hence,
$$
p_X(t)\leq\frac{(1-t)^n}{(1-x)^n}p_X(x).
$$
Inequality \eqref{eq:small_exact} follows as a direct computation shows that for every $1-\frac{1}{n}\leq x\leq 1$,
\begin{equation*}
p_X(x)=\frac{(n(1-x)/2)^n}{n!}.
\end{equation*}
We proceed to prove \eqref{eq:chernoff}.
For every $t$ and every $\mu>0$, by Chernoff's bound,
\begin{align*}p_X(t)&={\rm Pr}\left(\sum_{i=1}^n U_i>nt\right)\leq e^{-\mu n t}\prod_{i=1}^n\EE\left[e^{-\mu U_i}\right]\\
&=e^{-\mu n t}\left(\frac{1}{2}\int_{-1}^1 e^{-\mu x}\d x\right)^n=e^{-\mu n t}\left(\frac{e^{\mu}-e^{-\mu}}{2\mu}\right)^n.
\end{align*}
For every $\mu> 0$,
$$\frac{e^{\mu}-e^{-\mu}}{2\mu}=\sum_{k=0}^{\infty}\frac{1}{(2k+1)!}\mu^{2k}<\sum_{k=0}^{\infty}\frac{1}{k!}\left(\frac{\mu^2}{6}\right)^k=e^{\frac{\mu^2}{6}},$$
and hence for every $t$, 
$$p_X(t)\leq e^{-\mu n t}\left(\frac{e^{\mu}-e^{-\mu}}{2\mu}\right)^n<  e^{-\mu n t}e^{\frac{\mu^2}{6}n}.$$
Therefore, for every $t>0$, by taking $\mu=3t$,
\begin{equation*}p_X(t)< e^{-3nt^2}e^{3nt^2/2}=e^{-3nt^2/2}.
\qedhere\end{equation*}
\end{proof}
\bibliographystyle{amsplain}
\bibliography{ref}

\end{document}